\documentclass[12pt,twoside]{amsart}
\usepackage{amsfonts}
\usepackage{amsmath}
\usepackage{amssymb}
\usepackage{graphicx}
\usepackage{mathrsfs}
\usepackage{tikz}
\usepackage{lscape}

\theoremstyle{plain}
\newtheorem{thm}{Theorem}[section]
\newtheorem{cor}[thm]{Corollary}

\newtheorem{lem}[thm]{Lemma}
\newtheorem{prop}[thm]{Proposition}

\newtheorem{defn}{Definition}[section]

\newtheorem{rem}{Remark}

\usepackage{afterpage}

\usepackage[english]{babel}

\usepackage{fancyhdr}

\numberwithin{equation}{section}

\def\1{{\mathchoice {\rm 1\mskip-4mu l} {\rm 1\mskip-4mu l}{\rm 1\mskip-4.5mu l} {\rm 1\mskip-5mu l}}}
\newcommand{\ds}{\displaystyle}

\usepackage[width=16.00cm, height=21.00cm, left=2.00cm, right=2.00cm, top=4cm, bottom=2cm]{geometry}

\title{Modified Bergman spaces on the unit ball of $\mathbb C^n$ and applications}

\author[H. Ben Amor and N. Ghiloufi]{Hajer Ben Amor and Noureddine Ghiloufi}
\email{ben.amor.hajer12@gmail.com, noureddine.ghiloufi@fsg.rnu.tn}
\address{University of Gabes\\ Faculty of Sciences of Gabes\\ LR17ES11 Mathematics and Applications laboratory\\ 6072, Gabes, Tunisia.}

\begin{document}

	\begin{abstract}
		In this paper, we introduce new spaces of holomorphic functions on the unit ball $\mathbb{B}_{n}$ of $\mathbb{C}^{n}$ generalizing the classical Bergman spaces. The main results include the properties of some operators and integrals representations such as Bergman-type projections, and Berezin transform.	\end{abstract}
	\maketitle
	\section{Introduction and Preliminary results}
	The early 1970s marked the beginning of Bergman space notion, which found more and more applications in mathematics, especially in complex analysis of both one and several complex variables, and geometry. A fundamental concept is the reproducing kernel for Hilbert spaces, which is named after Stefan Bergman. Moreover, the above notation plays an essential role in the topological properties of Bergman spaces.\\ In recent years, the classical Bergman spaces have been extended to what we call modified Bergman spaces, so in this paper, we will study some properties of these spaces for the unit ball of $\mathbb{C}^{n}$ by introducing a new set of holomorphic functions related to a specific probability measure $d\mu_{a,b}$, for $a>-1$ and $b>-n$, such that if we take $b=0$, we re-obtain the classical case developed in  \cite{Zh}. \\

	This paper is divided into five parts. In the first, we introduce some notations and results that we will use in the rest of the paper. Then, we present some preliminary results on the reproducing kernel $\mathbb{K}_{a,b}$ of the modified Bergman space $\mathcal{A}^{2}_{a,b}(\mathbb B_n)$, and the boundedness of the projection from $L^{p}(\mathbb B_n,d\mu_{a,b})$ onto $\mathcal{A}^{p}_{a,b}(\mathbb B_n)$ for $1\leq p < + \infty$. In the 3rd section, we provide some properties for the Berezin transform $\mathbf{B}_{a,b}$, which was originally introduced by Berezin to solve certain physical problems. In the 4th section, we review the Bergman-Poincar\'e metric and introduce the concept of BMO (bounded mean oscillation) functions. In the last section, we investigate the location of the zeros of the hypergeometric function $Q_{a,b}$ by considering some special cases of the parameters $a$ and $b$.  \\
	
	Throughout the paper, for every  $z=(z_1,z_2,\dots,z_n),\ w=(w_1,w_2,\dots,w_n)\in\mathbb C^n$ we set $\langle z,w\rangle=\sum_{j=1}^nz_j\overline{w}_j$ the Euclidean inner product on $\mathbb C^n$ and $|z|=\sqrt{\langle z,z\rangle}$ the Euclidean norm. We denote by $B(z_0,r_0)$  the ball of $\mathbb C^n$ with center $z_0$ and radius $r_0>0$. i.e. $B(z_0,r_0):=\{z\in\mathbb C^n,\ |z-z_0|<r_0\}.$ For the unit ball, we use $\mathbb B_n$ instead of $B(0,1)$. If $z=(z_1,z_2,\dots,z_n)\in\mathbb C^n$ and $\alpha=(\alpha_1,\dots,\alpha_n)\in \mathbb{N}^{n}$ is a multi-indexes of non-negative integers ($\alpha_j\geq0$), we set $|\alpha|=\alpha_1+\dots+\alpha_n,\ \alpha!=\alpha_1!\dots\alpha_n!$ and $z^\alpha=z_1^{\alpha_1}\dots z_n^{\alpha_n}$.\\
	
	Many results will be described in terms of the  hypergeometric functions $\ _qF_s$ defined by \begin{equation*}
		\ _qF_s \left(\begin{matrix} c_1,\dots,c_q\\
			d_1,\dots,d_s \end{matrix} \middle | z \right) = \sum_{n=0}^{+\infty} \frac{(c_1)_{n}\dots (c_q)_{n}}{(d_1)_{n}\dots(d_s)_{n}} \frac{z^{n}}{n!}
	\end{equation*}
	where $(c)_{n}$ is the Pochhammer symbol given by $(c)_n=1$ if $n=0$ and $(c)_n=c(c+1)\dots(c+n-1)$ for every $n\geq 1$. We claim that for every $c\in\mathbb C\smallsetminus\mathbb Z^-$,  $(c)_{n}=\frac{\Gamma(c+n)}{\Gamma(c)}$.\\
	
	For every $a>-1$ and $b>-n$, we consider the positive measure $	d\mu_{a,b}(z)$ on $\mathbb{B}_{n} $ defined by: \begin{equation*}
		d\mu_{a,b}(z)= \frac{(n-1)!}{\pi^{n}\beta(a+1,b+n)} \lvert z | ^{2b} (1-\lvert z |^{2})^{a} d\lambda_{n}(z),
	\end{equation*}
	where $\lambda_n$ is the Lebesgue measure on $\mathbb C^n$ and $\beta$ is the beta-function defined by \begin{equation*}
		\beta(s,t) = \int_{0}^{1} x^{s-1}(1-x)^{t-1} dx=\frac{\Gamma(s)\Gamma(t)}{\Gamma(s+t)} \ \ \forall s,t>0.
	\end{equation*}
	The constants above are chosen so that  $d\mu_{a,b}(z)$ becomes a probability measure on $\mathbb{B}_{n} $.\\
	For every  $0<p<+\infty$, we denote by $\mathcal{A}^{p}_{a,b}(\mathbb{B}_{n})$ the Bergman-type space defined as the set of holomorphic functions on $\mathbb{B}_{n}$ that belong to the space
	\begin{equation*}
		L^{p}(\mathbb{B}_{n}, d\mu_{a,b})=\{f: \mathbb{B}_{n} \to \mathbb{C} \text{ measurable function  such that }\| f \|_{a,b,p}< + \infty\},
	\end{equation*}
	where \begin{equation*}
		\| f \|^{p}_{a,b,p}:= \int_{\mathbb{B}_{n}} |f(z)|^{p} d\mu_{a,b}(z).
	\end{equation*}
	If there is no ambiguity, we denote $\| f \|_p$ instead of $\| f \|_{a,b,p}$.
	\begin{rem}
		Suppose $\alpha=(\alpha_1,\dots,\alpha_n)\in \mathbb{N}^{n}$ is a multi-index of nonnegative integers. Then \begin{equation}\label{eq1}
				\int_{\mathbb{S}^{n}} \lvert \xi ^{\alpha} \rvert^{2} d\sigma_{n}(\xi)= \frac{2 \pi^{n} \alpha!}{(\lvert \alpha \rvert +n-1)!}
		\end{equation}
	\end{rem}
	\begin{prop}
		For any $0<r_{0}<1$, there exists $c(r_{0})=c_{a,b,n}(r_{0})$ such that for any $0<p<+\infty$ and $f \in \mathcal{A}^{p}_{a,b}(\mathbb{B}_{n})$, we have
		\begin{equation*}
			\sup_{z\in \overline{B}(0,r_0)}| f(z)|^{p} \leq \frac{n\beta(a+1,b+n)}{c(r_{0})} \| f \|_{p}^p.
		\end{equation*}
		It is possible to choose
		\begin{equation*}
			c(r_{0})= r_{1}^{2n} \mathfrak a(r_{0}) \mathfrak b(r_{0}) \quad \text{with} \quad r_{1}=\frac{1}{2}\min(r_{0},1-r_{0}),
		\end{equation*}
		$$\mathfrak a(r_{0}) := \begin{cases} \big(1-(r_{0}+r_{1})^{2}\big)^{a} & \text{if}\quad a \geq 0 \\
			\big(1-(r_{0}-r_{1})^{2}\big)^{a} & \text{if}\quad -1<a<0. \end{cases}$$
		and
		$$ \mathfrak b(r_{0}) := \begin{cases} (r_{0}-r_{1})^{2b} & \text{if}\quad b \geq 0, \\
			(r_{0}+r_{1})^{2b} & \text{if}\quad -n<b<0. \end{cases}$$
	\end{prop}
	\begin{proof}
		Let $0<r_0<1$, $0<p<+\infty$ be two fixed real numbers and  $f \in \mathcal{A}^{p}_{a,b}(\mathbb{B}_{n})$. Then there exists $z_0$ with $|z_0|=r_0$ such that  $$\sup_{B(0,r_{0})} |f|^{p}=  |f(z_{0})|^{p}.$$
		In view of the sub-harmonicity of $|f|^{p}$, we obtain
		\begin{equation*}
			|f(z_{0})|^{p} \leq \frac{n!}{\pi^{n} r_{1}^{2n}} \int_{B(z_{0},r_{1})} |f(w)|^{p} d\lambda_{n}(w), \quad \text{where}\quad  r_{1}=\frac{1}{2}\min(r_{0},1-r_{0}),
		\end{equation*}
		then \begin{equation*}
			|f(z_{0})|^{p} \leq \frac{n \beta(a+1,b+n)}{ r_{1}^{2n}} \int_{B(z_{0},r_{1})} \frac{|f(w)| ^{p}}{|w|^{2b} (1-|w|^{2})^{a}} d\mu_{a,b}(w).
		\end{equation*}
		If $w \in B(z_{0},r_{1})$, then $r_{0}-r_{1}< |w|<r_{0}+r_{1}$. Thus we obtain
		$$|w|^{2b} \geq \mathfrak b(r_{0}) := \begin{cases} (r_{0}-r_{1})^{2b} & \text{if $b \geq 0$,} \\
			(r_{0}+r_{1})^{2b} & \text{if\ } -n<b<0. \end{cases}$$
		and
		$$(1-| w|^{2})^{a} \geq \mathfrak a(r_{0}) := \begin{cases} \big(1-(r_{0}-r_{1})^{2}\big)^{a} & \text{if\ }a \geq 0, \\
			\big(1-(r_{0}+r_{1})^{2}\big)^{a} & \text{if } -1<a<0. \end{cases}$$
		Hence \begin{align*}
			|f(z_{0})|^{p}  & \leq \frac{n \beta(a+1,b+n)}{r_{1}^{2n} \mathfrak a(r_{0})\mathfrak b(r_{0})} \int_{B(z_{0},r_{1})} |f(w)|^{p} d\mu_{a,b}(w) \\ \nonumber
			\ & \leq   \frac{n \beta(a+1,b+n)}{r_{1}^{2n} \mathfrak a(r_{0})\mathfrak b(r_{0})} \|f\|_{p}^{p}.
		\end{align*}
	\end{proof}
	\begin{cor}
		For every $a>0,\ b>-n$ and $0<p<+\infty$, the space $\mathcal{A}_{a,b}^{p}(\mathbb{B}_{n})$ is closed in $L^{p}(\mathbb{B}_{n},d\mu_{a,b})$, and for any $z \in \mathbb{B}_{n}$, the linear form $\delta_{z}: \mathcal{A}_{a,b}^{p}(\mathbb{B}_{n}) \longrightarrow \mathbb{C}$ defined by $\delta_{z}(f)=f(z)$ is bounded on $\mathcal{A}_{a,b}^{p}(\mathbb{B}_{n})$.
	\end{cor}
	\begin{proof}
		Let $(f_{n})_{n} \subset  \mathcal{A}_{a,b}^{p}(\mathbb{B}_{n})$ be a sequence  that converges to $f \in L^{p}(\mathbb{B}_{n},d\mu_{a,b})$. Thanks to the previous proposition, the sequence $(f_{n})_{n}$ converges to $f$ uniformly on every compact subset of $\mathbb{B}_{n}$. Hence, the function $f$ is holomorphic on $\mathbb{B}_{n}$, and we conclude that $f \in  \mathcal{A}_{a,b}^{p}(\mathbb{B}_{n})$. In addition, we have from the previous proposition: \begin{equation*}
			|\delta_{z}(f)| = |f(z)| \leq C \|f\|_{p}
		\end{equation*}
		for every $f \in  \mathcal{A}_{a,b}^{p}(\mathbb{B}_{n})$. This shows the continuity of the Dirac form $\delta_{z}$ on  $\mathcal{A}_{a,b}^{p}(\mathbb{B}_{n})$.
		
	\end{proof}

	\section{Reproducing Kernel of $\mathcal A^{2}_{a,b}(\mathbb{B}_{n})$}
	Since   $\mathcal{A}_{a,b}^{2}(\mathbb{B}_{n})$ is closed in the Hilbert space $L^{2}(\mathbb{B}_{n},d\mu_{a,b})$ then we conclude   that $\mathcal{A}_{a,b}^{2}(\mathbb{B}_{n})$ is also a Hilbert space. Moreover, as the Dirac operator $\delta_z$ is continuous on $\mathcal{A}_{a,b}^{2}(\mathbb{B}_{n})$ then thanks to the Riesz representation theorem, we deduce that for every $z\in \mathbb B_n$, there exists $\mathbb K_{a,b}(.,z)\in \mathcal{A}_{a,b}^{2}(\mathbb{B}_{n})$ such that for every $f\in \mathcal{A}_{a,b}^{2}(\mathbb{B}_{n})$, $\delta_z(f)=\langle f,\mathbb K_{a,b}(.,z)\rangle_{a,b}$. i.e.
	\begin{equation*}
		f(z)= \int_{\mathbb{B}_{n}} f(w) \overline{\mathbb{K}_{a,b}(w,z)} d\mu_{a,b}(w).
	\end{equation*}
	The function $\mathbb{K}_{a,b}$ is called the reproducing (Bergman) kernel of $ \mathcal{A}^{2}_{a,b}(\mathbb B_n)$.
	
	\begin{thm}	
		The reproducing Bergman Kernel $\mathbb{K}_{a,b} $ of $\mathcal{A}_{a,b}^{2}(\mathbb{B}_{n})$ is given by $\mathbb{K}_{a,b}(w,z)=\mathcal{K}_{a,b}(\langle z,w\rangle)$ where
		$$\mathcal{K}_{a,b}(\xi) = \ _2F_1 \left (\begin{matrix} n,\  a+b+n+1 \\
			b+n \end{matrix} \middle | \xi\right ) =\frac{1}{(1- \xi )^{a+n+1}} \ _2F_1 \left (\begin{matrix} b,\ -(a+1) \\
			b+n \end{matrix} \middle | \xi\right ). $$
	\end{thm}
	\begin{proof}
		The last equality is due to a formula on page 47 in \cite{Ma-Ob-So}. For the first one, we claim that if we set
		\begin{equation*}
			e_{\alpha}(z)=  \sqrt{\frac{(|\alpha| +n-1)! \Gamma(b+n) \Gamma(|\alpha| +b+n+a+1) }{(n-1)!\alpha! \Gamma(a+b+n+1)\Gamma(|\alpha| +b+n)}} z^{\alpha}.
		\end{equation*}
		for every $\alpha \in \mathbb{N}^{n}$, the sequence  $\big(e_{\alpha}\big)_{\alpha \in \mathbb{N}^{n}}$ is a Hilbert basis of  $\mathcal{A}_{a,b}^{2}(\mathbb{B}_{n})$.  Hence the reproducing kernel of $\mathcal{A}^{2}_{a,b}(\mathbb{B}_{n})$ is given by
		\begin{align*}
			\mathbb{K}_{a,b}(w,z)& = \sum_{k=0}^{+\infty} \sum_{|\alpha|=k} e_{\alpha}(w) \overline{e_{\alpha}(z)} \\
			& =  \sum_{k=0}^{+\infty} \sum_{|\alpha|=k}  \frac{(|\alpha|+n-1)! \Gamma(b+n) \Gamma(| \alpha | +b+n+a+1) }{(n-1)!\alpha! \Gamma(a+b+n+1)\Gamma(| \alpha | +b+n)} w^{\alpha} \overline{z^{\alpha}} \\
			\ & =  \sum_{k=0}^{+\infty} \frac{(k +n-1)!\Gamma(b+n)}{(n-1)! \Gamma (a+b+n+1)} \frac{ \Gamma(k +b+n+a+1) }{\Gamma(k +b+n)k!} \sum_{|\alpha|=k} \frac{k!}{\alpha!} w^{\alpha} \overline{z^{\alpha}}. \\
		\end{align*}
		Since \begin{equation*}
			\sum_{|\alpha|=k} \frac{k!}{\alpha!} w^{\alpha} \overline{z^{\alpha}} = \langle w,z \rangle^{k},
		\end{equation*}
		then \begin{align*}
			\mathbb{K}_{a,b}(w,z)& = \sum_{k=0}^{+\infty} \frac{(n)_{k}(a+b+n+1)_{k}}{(b+n)_{k}} \frac{\langle w,z \rangle^{k}}{k!} \\ \nonumber
			\ & = {}_2F_1 \left (\begin{matrix} n,\ a+b+n+1 \\
				b+n \end{matrix} \middle | \langle w,z \rangle \right )\\
			&=:\mathcal K_{a,b}(\langle w,z \rangle).
		\end{align*}
	\end{proof}
	\begin{cor}
		If $\mathbb{P}_{a,b}$ denotes the orthogonal projection from $L^{2}(\mathbb{B}_{n},d\mu_{a,b})$ onto $\mathcal{A}^{2}_{a,b}(\mathbb{B}_{n})$ then for every $f \in L^{2}(\mathbb{B}_{n},d\mu_{a,b})$,
		\begin{equation}\label{e1}
			\mathbb{P}_{a,b} f(z)= \int_{\mathbb{B}_{n}} f(w)\mathcal K_{a,b}(\langle z,w\rangle) d\mu_{a,b}(w).
		\end{equation}
	\end{cor}
	\begin{rem}
		Using the density of $L^2(\mathbb{B}_{n},d\mu_{a,b})$ in $L^1(\mathbb{B}_{n},d\mu_{a,b})$, one can deduce that the integral formula \eqref{e1}, defines $\mathbb P_{a,b}$ as a linear operator on $L^1(\mathbb{B}_{n},d\mu_{a,b})$.
	\end{rem}
	Now, we will  establish the boundedness of the Bergman projection $\mathbb{P}_{a,b}$ on certain $L^{p}$ spaces, and for this, we need the following lemmas:
	\begin{lem}
		For every $-1<a<+\infty$ and $b>-n$, we set \begin{align*}
			I_{s}(z)&=\int_{\mathbb{B}_{n}} \frac{(1-| w |^{2})^{a} | w |^{2b}}{| 1- \langle z,w \rangle |^{a+n+1+s}} d\lambda_{n}(w) \\ \nonumber
			\ & = \frac{\pi^{n}}{(n-1)!} \beta(a+1,b+n) \ _3F_2 \left (\begin{matrix} \gamma,\  \gamma,\ b+n \\
				n,\  a+b+n+1 \end{matrix} \middle |  |z|^{2} \right )
		\end{align*}
		where $\gamma=(a+n+1+s)/2$. In particular, we have the following asymptotic properties \begin{enumerate}
			\item  If  $s < 0$, $I_{s}$ is  bounded  in $\mathbb{B}_{n}$.
			\item  If $s=0$, then \begin{equation*}
				I_{s}(z) \approx \log \bigg (\frac{1}{1-|z|^{2}}\bigg), \qquad 	as \  |z| \to 1^{-}.
			\end{equation*}
			\item  If $s>0$, then \begin{equation*}
				I_{s}(z) \approx \frac{1}{(1-|z|^{2})^{s}} , \qquad 	as \  |z| \to 1^{-}.
			\end{equation*}
		\end{enumerate}
	\end{lem}
        Here we use the notation $\theta_{1} \approx \theta_{2}$ to indicate that there exists $0<c_{1}<c_{2}$ such that we have $c_{1} \theta_{1}(z) \leq \theta_{2} \leq c_{2} \theta_{1}(z) $. i.e. $\theta_1\lesssim \theta_2$ and  $\theta_2\lesssim \theta_1$.
	\begin{proof}
		Let $\gamma=(a+n+1+s)/2$, then \begin{equation*}
			\frac{1}{| 1- \langle z,w \rangle |^{a+n+1+s}} = \sum_{k,j=0}^{+\infty} \frac{\Gamma(k+\gamma)}{k! \Gamma(\gamma)} \frac{\Gamma(j+\gamma)}{j! \Gamma(\gamma)} \langle z,w \rangle^{k} \langle w,z \rangle^{j}.
		\end{equation*}
		Hence \begin{equation*}
			I_{s}(z)= \int_{\mathbb{B}_{n}} (1-| w |^{2})^{a} | w |^{2b} \sum_{k,j=0}^{+\infty} \frac{\Gamma(k+\gamma)}{k! \Gamma(\gamma)} \frac{\Gamma(j+\gamma)}{j! \Gamma(\gamma)} \langle z,w \rangle^{k} \langle w,z \rangle^{j} d\lambda_{n}(w),
		\end{equation*}	
		we integrate in polar coordinates to obtain \begin{equation*}
			I_{s}(z)= \sum_{k,j=0}^{+\infty} \frac{\Gamma(k+\gamma)}{k! \Gamma(\gamma)} \frac{\Gamma(j+\gamma)}{j! \Gamma(\gamma)} \int_{0}^{1} (1-r^{2})^{a} r^{2b+2n-1+k+j} dr \int_{\mathbb{S}^{n}} \langle z, \xi \rangle^{k} \langle \xi , z \rangle^{j} d\sigma_{n}(\xi).
		\end{equation*}
		We assume that $U \xi=Y$, where $U$ is a unitary matrix such that the first coordinate of $Y$ is $y_{1}=\frac{\langle \xi,z\rangle}{|z|}$. Then \begin{equation*}
			I_{s}(z) =  \sum_{k,j=0}^{+\infty} \frac{\Gamma(k+\gamma)}{k! \Gamma(\gamma)} \frac{\Gamma(j+\gamma)}{j! \Gamma(\gamma)} |z|^{j}|z|^{k} \int_{0}^{1} (1-r^{2})^{a} r^{2b+2n-1+k+j} dr  \int_{\mathbb{S}^{n}} y_{1}^{j} \overline{y_{1}}^{k} d\sigma_{n}(Y).
		\end{equation*}
	Using Equation \eqref{eq1}, one can see that \begin{equation*}
			\int_{\mathbb{S}^{n}} y_{1}^{j} \overline{y_{1}}^{k} d\sigma_{n}(Y) = \frac{2 \pi^{n} k!}{(k +n-1)!}\delta_{k,j},
		\end{equation*}
		with $\delta_{k,j}$ is the Kronecker delta symbol given by $\delta_{k,j}=1$ or $0$ according to $k=j$ or $k \neq j$. It follows that \begin{align*}
			I_{s}(z)& = \sum_{k=0}^{+\infty} \frac{\Gamma(k+\gamma)^{2}}{k! \Gamma(\gamma)^{2}}  \frac{2\pi^{n}}{(n-1+k)!} | z |^{2k} \int_{0}^{1} (1-r^{2})^{a} r^{2b+2n+2k-1}dr \\ \nonumber
			\ & = \frac{\pi^{n}}{(n-1)!} \sum_{k=0}^{+\infty}  \frac{\Gamma(k+\gamma)^{2}}{k! \Gamma(\gamma)^{2}}  \frac{(n-1)!}{(n-1+k)!} \frac{\Gamma(a+1)\Gamma(b+n+k)}{\Gamma(a+1+b+n+k)} | z |^{2k} \\ \nonumber
			\ & = \frac{\pi^{n}}{(n-1)!} \frac{\Gamma(a+1)\Gamma(b+n)}{\Gamma(a+1+b+n)} \sum_{k=0}^{+\infty} \frac{(\gamma)_{k} (\gamma)_{k}}{(n)_{k}} \frac{(b+n)_{k}}{(a+b+n+1)_{k}} \frac{| z |^{2k}}{k!} \\ \nonumber
			\ & = \frac{\pi^{n}}{(n-1)!} \frac{\Gamma(a+1)\Gamma(b+n)}{\Gamma(a+1+b+n)} \ {}_3F_2 \left (\begin{matrix} \gamma, & \gamma, & b+n \\
				&  	n, & a+b+n+1 \ \ \ \end{matrix} \middle | | z |^{2} \right ).
		\end{align*}
		Now, by applying Stirling's formula, we conclude that \begin{align*}
			I_{s}(z) & \approx \sum_{k=0}^{+\infty} (k+1)^{2 \gamma +b+n -(a+b+n+1)-n-1} | z |^{2k} \\ \nonumber
			\ & \approx \sum_{k=0}^{+\infty} (k+1)^{2 \gamma -a-n-2} | z |^{2k} \\
            & \approx\sum_{k=0}^{+\infty} (k+1)^{s-1} | z |^{2k},
		\end{align*}
		as $| z | \to 1^{-}$. This completes the proof of the Lemma.
	\end{proof}
	Using the same techniques used in the previous lemma, one can prove the following lemma:
	\begin{lem}
		For every $-1<a,c<+\infty$ and $d,b>-n$, we set \begin{equation*}
			J_{s}(z)= \int_{\mathbb{B}_{n}} \frac{ | Q_{a,b}(\langle z,w \rangle ) | (1-| w |^{2})^{c} | w |^{2d}}{| 1- \langle z,w \rangle |^{c+n+1+s}} d\lambda_{n}(w)\quad \text{where}\quad Q_{a,b}(\xi ) = {}_2F_1 \left (\begin{matrix} b,\  -(a+1) \\
				b+n \end{matrix} \middle | \xi \right ).
		\end{equation*}	
		If $Q_{a,b}$ has no  zero \footnote{ Some comments about this technical assumption will be added at the end of the paper.} in the closed unit disk $\overline{\mathbb D}$ of $\mathbb C$ then $$J_{s}(z) \approx \begin{cases} 1 & \text{if }s < 0, \\
			\ds\log \left(\frac{1}{1-| z |^{2}}\right)	 & \text{if } s=0 \\
			\ds\frac{1}{(1-| z | ^{2})^{s}}  & \text{if }s>0 \end{cases}$$
		as $| z | \to 1^{-}$.
	\end{lem}

	\begin{thm} \label{Theorem 2.5}
		Let $-1<a,c<+\infty$ and $d,b>-n$.  We assume that $Q_{a,b}$ has no zero in $\overline{\mathbb{D}}$. We define the two integral operators $T$ and $S$ by  \begin{align*}
			Tf(z) & = \int_{\mathbb{B}_{n}} f(w)  \mathcal{K}_{a,b}( \langle z,w \rangle ) d\mu_{a,b}(w) \\ \nonumber
			\ & = \frac{\beta(c+1,d+n)}{\beta(a+1,b+n)} \int_{\mathbb{B}_{n}} \frac{f(w)Q_{a,b}( \langle z,w \rangle )}{(1- \langle z,w \rangle )^{a+n+1}} (1-| w |^{2})^{a-c} | w |^{2b-2d} d\mu_{c,d}(w),
		\end{align*}
		and \begin{equation*}
			Sf(z) =\int_{\mathbb{B}_{n}} \frac{f(w) | Q_{a,b}( \langle z,w \rangle )|}{ | 1- \langle z,w \rangle |^{a+n+1}} (1-| w |^{2})^{a-c} | w |^{2b-2d} d\mu_{c,d}(w).
		\end{equation*}
		Then for $1 \leq p < + \infty$, the following assertions are equivalent:
		\begin{enumerate}
			\item $T$ is bounded on  $L^{p}(\mathbb{B}_{n},d\mu_{c,d})$.
			\item  $S$ is bounded on  $L^{p}(\mathbb{B}_{n},d\mu_{c,d})$.
			\item The following conditions $(C)$ are satisfied \begin{equation*}
				(C)	\quad c+1<p(a+1) \qquad \text{and} \qquad   \begin{cases} d+n<p(b+n) & \mbox{if } p > 1, \\
						d \leq b & \mbox{if }p=1 \end{cases}
			\end{equation*}
			
		\end{enumerate}
	\end{thm}
	\begin{proof}
		It's clear that the boundedness of $S$ on  $L^{p}(\mathbb{B}_{n},d\mu_{c,d})$ implies the boundedness of $T$ on  $L^{p}(\mathbb{B}_{n},d\mu_{c,d})$. To prove that if $T$ is bounded  on $L^{p}(\mathbb{B}_{n},d\mu_{c,d})$ then $S$ is bounded too, it suffices to use the following transformation: \begin{equation*}
			\psi_{z}f(w)= \frac{(1- \langle z,w \rangle )^{a+n+1}  | Q_{a,b}( \langle z,w \rangle )|}{| 1- \langle z,w \rangle |^{a+n+1} Q_{a,b}( \langle z,w \rangle )}f(w).
		\end{equation*}
		Now, we assume that  $S$ is bounded on  $L^{p}(\mathbb{B}_{n},d\mu_{c,d})$. \\
		If $1<p<+\infty$, and $\frac{1}{p}+\frac{1}{q}=1$, the boundedness of $S$ on $L^{p}(\mathbb{B}_{n},d\mu_{c,d})$ is equivalent to the boundedness of the adjoint operator of $S$ on $L^{q}(\mathbb{B}_{n},d\mu_{c,d})$. It is easy to see that \begin{equation*}
			S^{*}f(z)= (1-| z |^{2})^{a-c} | z |^{2b-2d} \int_{\mathbb{B}_{n}} \frac{f(w)| Q_{a,b}( \langle z,w \rangle )|}{| 1- \langle z,w \rangle |^{a+n+1}} d\mu_{c,d}(w).
		\end{equation*}
		We apply $S^{*}$ to the function $f_{N}=(1-|z|^{2})^{N}$, for $N$ sufficiently large, so we obtain that \begin{equation*}
			\| S^{*}f_{N} |_{c,d,q}^{q} = \int_{\mathbb{B}_{n}} M  (1-| z |^{2})^{q(a-c)+c} | z |^{q(2b-2d)+2d} \times \bigg (
			\int_{\mathbb{B}_{n}} \frac{| Q_{a,b}( \langle z,w \rangle )| (1-| w |^{2})^{c+N} | w |^{2d}}{| 1- \langle z,w \rangle |^{a+n+1}} d\lambda_{n}(w) \bigg)^{q} d\lambda_{n}(z)
		\end{equation*}
		is finite, where $M= \bigg  ( \frac{(n-1)!}{\pi^{n}\beta(c+1,d+n)} \bigg)^{q+1}$. So according to the previous Lemma, we conclude that \begin{equation*}
			c+1<p(a+1) \ and \ d+n<p(b+n)
		\end{equation*}
		If $p=1$, $S$ is bounded on $L^{1}(\mathbb{B}_{n},d\mu_{c,d})$. This gives that $S^{*}$ is bounded on $L^{\infty}(\mathbb{B}_{n},d\mu_{c,d})$. By applying $S^{*}$ to the constant function $f\equiv 1$, we find \begin{equation*}
			\underset{z \in \mathbb{B}_{n}}{\sup}  (1-| z |^{2})^{a-c} | z |^{2b-2d} \int_{\mathbb{B}_{n}} \frac{ | Q_{a,b}( \langle z,w \rangle )| (1-| w |^{2})^{c} | w |^{2d}}{| 1- \langle z,w \rangle |^{a+n+1}} d\lambda_{n}(w)< + \infty.
		\end{equation*}
		Then, by applying the previous Lemma, we find that $a-c>0$, and $b-d\geq 0$. \\
		It remains to prove that $(3)$ implies $(2)$. We first consider the case $p=1$. Assume that \begin{equation*}
			a-c >0 \ \ and \ \ d\leq b.
		\end{equation*}
		We have \begin{align*}
			\| Sf \|_{c,d,1} & \leq \int_{\mathbb{B}_{n}} (1-|z|^{2})^{c} |z|^{2d} \int_{\mathbb{B}_{n}} \frac{|f(w)|(1-|w|^{2})^{2}|w|^{2b}}{|1-\langle z,w \rangle |^{a+n+1}} d\lambda_{n}(w) d\lambda_{n}(z) \\ \nonumber
			\ & \leq \int_{\mathbb{B}_{n}} |f(w)| (1-|w|^{2})^{a} |w|^{2b} \int_{\mathbb{B}_{n}} \frac{(1-|z|^{2})^{c} |z|^{2d}}{|1-\langle z,w \rangle |^{a+n+1}} d\lambda_{n}(z) d\lambda_{n}(w) \qquad (Fubini).
		\end{align*}
		Using the previous lemma, we find that \begin{equation*}
			\| Sf \|_{c,d,1}  \leq M \| f \|_{c,d,1}
		\end{equation*}
		where $M $ is a constant. So the boundedness of $S$ on  $L^{1}(\mathbb{B}_{n},d\mu_{c,d})$. \\
		For $p>1$, we appeal to Schur's test,  such that the details can be found extensively in \cite{Zh} page 45. Thus we assume that  $1<p<+\infty$, and a positive function $h(z)$ on $\mathbb{B}_{n}$ that will satisfy the assumptions
		in Schur's test, which. It turns out that such a function exists in the form \begin{equation*}
			h(z)= \frac{1}{(1-|z|^{2})^{s} |z|^{2t}}, \ \ \ \ s,t \in \mathbb{R}
		\end{equation*}
		In fact if we write \begin{equation*}
			Sf(z) =\int_{\mathbb{B}_{n}} \frac{f(w) | Q_{a,b}( \langle z,w \rangle )|}{ | 1- \langle z,w \rangle |^{a+n+1}} (1-| w |^{2})^{a-c} | w |^{2b-2d} d\mu_{c,d}(w)
		\end{equation*}
		then the conditions that the number s has to satisfy become \begin{align*}
			&	\int_{\mathbb{B}_{n}} \frac{ | Q_{a,b}( \langle z,w \rangle )|}{ | 1- \langle z,w \rangle |^{a+n+1}} (1-| w |^{2})^{a} | w |^{2b} h(w)^{q} d\lambda_{n}(w) \\ \nonumber
			\ & =  \int_{\mathbb{B}_{n}} \frac{ | Q_{a,b}( \langle z,w \rangle )|}{ | 1- \langle z,w \rangle |^{a+n+1}} (1-| w |^{2})^{a-qs} | w |^{2b-2qt} d\lambda_{n}(w) \\ \nonumber
			\ & \leq \frac{C_{1}}{(1-|z|^{2})^{sq} |z|^{2qt}}, \ \ z \in \mathbb{B}_{n}
		\end{align*}
		and \begin{align*}
			&	\int_{\mathbb{B}_{n}} \frac{ | Q_{a,b}( \langle z,w \rangle )|}{ | 1- \langle z,w \rangle |^{a+n+1}} (1-| w |^{2})^{a-c} | w |^{2b-2d} h(z)^{p} d\mu_{c,d}(z)\\ \nonumber
			\ & = (1-| w |^{2})^{a-c} | w |^{2b-2d} \int_{\mathbb{B}_{n}} \frac{ | Q_{a,b}( \langle z,w \rangle )|}{ | 1- \langle z,w \rangle |^{a+n+1}} (1-|z|^{2})^{c-sp} |z|^{2d-2pt} d\lambda_{n}(z) \\ \nonumber
			\ & \leq \frac{C_{2}}{(1-|w|^{2})^{sp}|w|^{2pt}}
		\end{align*}
		where $q$ is the conjugate exponent of $p$ and $C$ is some positive constant. According to the previous lemma, these estimates hold if \begin{equation*}
			0<s<\frac{a+1}{q}, \qquad 0<t<\frac{b+n}{q}
		\end{equation*}
		and \begin{equation*}
			\frac{c-a}{p} < s < \frac{c+1}{p}, \qquad\frac{d-b}{p} < t< \frac{d+n}{p}
		\end{equation*}
		Hypothesis $(3)$ gives that \begin{equation*}
			\left] 0, \frac{b+n}{q} \right[\cap  \left]\frac{d-b}{p},  \frac{d+n}{p} \right[\neq \emptyset
		\end{equation*}
		\begin{equation*}
			\left]\ 0, \frac{a+1}{q} \right[\cap \left]\frac{c-a}{p},\frac{c+1}{p} \right[\neq \emptyset
		\end{equation*}
		This shows that the desired $s,t$ exist, and so the operator $S$ is bounded on $L^{p}(\mathbb{B}_{n},d\mu_{c,d})$.
	\end{proof}
	\begin{rem}
		The condition that $Q_{a,b}$ does not vanish on the unit disk in the previous result is a technical assumption and the result may be true without it. This condition was used exactly in the proof of $(2)$ implies $(3)$.
	\end{rem}
	\begin{cor}
		Let $-1<a,c<+\infty$, $d,b>-n$ and $1\leq p < + \infty$. If the following conditions are satisfied \\ $	c+1<p(a+1)$ and \[ \begin{cases} d+n<p(b+n) & \text{if $p > 1$,} \\
			d \leq b & \text{if $p=1$} \end{cases}\] then the projection $	\mathbb{P}_{a,b}$ is  bounded  from $L^{p}(\mathbb{B}_{n},d\mu_{c,d})$ onto  $\mathcal{A}^{p}_{c,d}(\mathbb{B}_{n})$. \\
		Conversely, if $Q_{a,b}$ has no zero in $\overline{\mathbb D}$ and  $\mathbb{P}_{a,b}$ is bounded on $L^{p}(\mathbb{B}_{n},d\mu_{c,d})$ then the conditions $(C)$ are satisfied.
	\end{cor}
	\begin{proof}
		It's an immediate consequence of the previous Theorem.
	\end{proof}
	\begin{thm}
		For every $-1<a<+\infty$, $b>-n$ and $1 < p < + \infty$, the topological dual of $\mathcal{A}^{p}_{a,b}(\mathbb B_n)$ is the space $\mathcal{A}^{q}_{a,b}(\mathbb B_n)$ under the integral pairing \begin{equation*}
			\langle f,g \rangle_{a,b} = \int_{\mathbb{B}_{n}} f(z) \overline{g(z)} d\mu_{a,b}(z),\qquad\forall\; f \in \mathcal{A}^{p}_{a,b}(\mathbb B_n), \ \ g \in \mathcal{A}^{q}_{a,b}(\mathbb B_n),
		\end{equation*}
		where $\frac{1}{p}+\frac{1}{q}=1$.
	\end{thm}
	\begin{proof}
		If $g \in \mathcal{A}^{q}_{a,b}(\mathbb B_n)$, then using H\"older's inequality, one can show that the linear form defined by $g$:  \begin{equation*}
			f\longmapsto\langle f,g \rangle_{a,b} = \int_{\mathbb{B}_{n}} f(z) \overline{g(z)} d\mu_{a,b}(z)
		\end{equation*}
		is bounded on $\mathcal{A}^{p}_{a,b}(\mathbb B_n)$. Conversely, if $F$ is a bounded linear form on $\mathcal{A}^{p}_{a,b}(\mathbb B_n)$, then by Hahn-Banach extension theorem, $F$ can be extended to a bounded linear functional on $L^{p}(\mathbb{B}_{n},d\mu_{a,b})$, (still denoted $F$), by duality, there exists $h\in L^{q}(\mathbb{B}_{n},d\mu_{a,b})$ such that \begin{equation*}
			F(f)= \int_{\mathbb{B}_{n}} f(z) \overline{h(z)} d\mu_{a,b}(z), \quad \forall\; f \in L^{p}(\mathbb{B}_{n},d\mu_{a,b}).
		\end{equation*}
		The previous corollary guarantees the continuity of the projection from $L^{p}(\mathbb{B}_{n},d\mu_{a,b})$ onto $\mathcal{A}^{p}_{a,b}(\mathbb B_n)$	which implies that for every $f \in \mathcal{A}^{p}_{a,b}(\mathbb B_n)$, \begin{equation*}
			F(f)= \langle f,h \rangle_{a,b} = \langle \mathbb{P}_{a,b} f , h \rangle_{a,b} = \langle f , \mathbb{P}_{a,b}^{*} h \rangle_{a,b}.
		\end{equation*}
		If we set $g=\mathbb{P}_{a,b}^{*} h$, then $g \in \mathcal{A}^{q}_{a,b}(\mathbb B_n) $ and $F(f)= \langle f,g \rangle_{a,b}$ for every $f \in \mathcal{A}^{p}_{a,b}(\mathbb B_n)$.
	\end{proof}
	\section{Berezin Transform}
	We introduce some notations which will be useful in the sequel of this section. For $z \in \mathbb{B}_{n}$ we consider the automorphism $\phi_{z}$ given by \begin{equation}
		\phi_{z}(\zeta)=\frac{z-P_{z}(\zeta)-(1-|z|^{2})^{1/2}Q_{z}(\zeta)}{1-\langle \zeta , z \rangle}, \qquad \zeta \in \mathbb{B}_{n}
	\end{equation}
	where $P_{z}(\zeta)=\frac{\langle \zeta,z \rangle}{|z|^{2}}z$ is the orthogonal projection onto the subspace generated by $z$ if $z\neq0$, and $Q_{z}(\zeta)=\zeta-P_{z}(\zeta)$. If $z=0$ we take $\phi_0(\zeta)=-\zeta$. The automorphism $\phi_z$ verifies the following properties:
	\begin{itemize}
		\item $\phi_{z} \circ \phi_{z}(\zeta)=\zeta$, for every $\zeta \in \mathbb{B}_{n}$.
		\item  The real Jacobian determinant $J_{\mathbb R}(\phi_{z})$ is given by \begin{equation*}
			J_{\mathbb R}(\phi_{z})(\zeta) = \left( \frac{1-|z|^{2}}{| 1- \langle \zeta , z \rangle |^{2}} \right)^{n+1}
		\end{equation*}
		\item and the following identity:
        \begin{equation}\label{eq2}
			1- |\phi_{z}(\zeta) |^{2}=\frac{(1-|z|^{2})(1-|\zeta |^{2})}{|1- \langle \zeta , z  \rangle |^{2}}.
		\end{equation}
	\end{itemize}
	\begin{defn}
		For every $a>-1$, $b>-n$ and $f$  a measurable function on the unit ball $\mathbb{B}_{n}$, the Berezin transform $\mathbf{B}_{a,b}$ of $f$ is given by \begin{equation*}
			\mathbf{B}_{a,b}f(z)= \langle f \kappa_z,  \kappa_z\rangle_{a,b}=  \int_{\mathbb{B}_{n}} f(w) \frac{|\mathcal{K}_{a,b}(\langle w,z\rangle)|^2}{\mathcal{K}_{a,b}(|z|^{2})} d\mu_{a,b}(w),
		\end{equation*}
		where  $\kappa_z$ is the normalized reproducing kernel of $\mathcal{A}^{2}_{a,b}(\mathbb{B}_{n})$:
        \begin{equation*}
	\kappa_{z}(\zeta)=\frac{|\mathbb{K}_{a,b}(\zeta,z)|^{2}}{\| \mathbb{K}_{a,b}(.,z) \|_{a,b,2}^{2} }=\frac{\mathcal{K}_{a,b}(\langle \zeta,z\rangle)}{\sqrt{\mathcal{K}_{a,b}(|z|^{2})}}.
\end{equation*}
	\end{defn}
	This is because \begin{equation*}
		\| \mathbb{K}_{a,b}(.,z) \|_{a,b,2}^{2} = \mathbb{K}_{a,b}(z,z)=\mathcal K_{a,b}(|z|^2)=\frac{Q_{a,b}(|z|^{2})}{(1-|z|^{2})^{a+n+1}}.
	\end{equation*}
	One can see that $\mathbf{B}_{a,b}$ is well defined on $L^{1}(\mathbb{B}_{n}, d\mu_{a,b}).$
	\begin{prop}
		The Berezin transform $\mathbf{B}_{a,b}$ is one-to-one on $L^{1}(\mathbb{B}_{n}, d\mu_{a,b})$.
	\end{prop}
	\begin{proof}
		Let $f\in Ker(\mathbf{B}_{a,b})$, i.e. $f \in L^{1}(\mathbb{B}_{n},d\mu_{a,b})$ with  $\mathbf{B}_{a,b}f=0$. If we set $$
		H(z)=\mathcal{K}_{a,b}(|z|^{2})  \mathbf{B}_{a,b}f(z) = \int_{\mathbf{B}_{n}} f(w) \mathcal{K}_{a,b}( \langle w,z \rangle) \mathcal{K}_{a,b}(\langle z,w\rangle ) d\mu_{a,b}(w).$$
		then $H\equiv 0$ on $\mathbb{B}_{n}$ gives
		$$\frac{\partial ^{|\alpha |+|\gamma|}H}{\partial z^{\alpha}\partial\overline{z}^{\gamma}}(0)= \frac{(n)_{|\alpha |}(a+b+n+1)_{|\alpha |}}{(b+n)_{|\alpha |}} \frac{(n)_{|\gamma |}(a+b+n+1)_{|\gamma |}}{(b+n)_{|\gamma |}} \int_{\mathbb{B}_{n}} f(w) \overline{w^{\alpha}} w^{\gamma} d\mu_{a,b}(w)=0.$$
		It follows that$$\int_{\mathbb{B}_{n}} f(w) \overline{w^{\alpha}} w^{\gamma} d\mu_{a,b}(w)=0$$
		for every $ \alpha =(\alpha_{1},...,\alpha_{n}),\ \gamma=(\gamma_{1},...,\gamma_{n})\in\mathbb N^n$.
		The density of polynomials in $L^{1}(\mathbb{B}_{n},d\mu_{a,b})$ gives that $f=0$ and the proposition is proved.
	\end{proof}
	\begin{thm}
		Let $a,c>-1$ and $d,b>-n$. The Berezin transform $\mathbf{B}_{a,b}: L^{p}(\mathbb{B}_{n},d\mu_{c,d}) \longrightarrow L^{p}(\mathbb{B}_{n},d\mu_{c,d})  $ is bounded if and only if \\
		$$	c+1<p(a+1) \qquad and  \qquad \begin{cases} d+n<p(b+n) & \text{if $p > 1$,} \\
			d \leq b & \text{if $p=1$} \end{cases} $$
	\end{thm}
	\begin{proof}
		This is a direct consequence of Theorem \ref{Theorem 2.5} without the condition on the term $Q_{a,b}$.
	\end{proof}
	If we denote by $\mathcal{C}({\overline{\mathbb{B}}_{n}})$ the space of continuous functions on $\overline{\mathbb{B}}_{n}$ and $\mathcal{C}_{0}({\overline{\mathbb{B}}_{n}})$ the subspace of continuous functions on $\overline{\mathbb{B}}_{n}$ that vanish at the boundary of $\mathbb{B}_{n}$ then we have the following proposition:
	\begin{prop}
		If $f \in \mathcal{C}({\overline{\mathbb{B}}_{n}}) $ then  $\mathbf{B}_{a,b} f \in \mathcal{C}({\overline{\mathbb{B}}_{n}}) $ and $f-\mathbf{B}_{a,b}f \in \mathcal{C}_{0}(\overline{\mathbb{B}}_{n})$.
	\end{prop}
	\begin{proof}
		By making the change of variable $w=\phi_{z}(\zeta)$, we obtain
		$$\begin{array}{l}
			\mathbf{B}_{a,b}f(z)=\ds \frac{1}{\mathcal{K}_{a,b}(|z|^{2})} \int_{\mathbb{B}_{n}} f(w) | \mathcal{K}_{a,b}( \langle w,z \rangle ) |^{2} d\mu_{a,b}(w)\\
			=\ds \frac{ (1-|z|^{2})^{a+n+1}}{Q_{a,b}(|z|^{2})} \int_{\mathbb{B}_{n}} f(w) \frac{|Q_{a,b}( \langle w,z \rangle  )|^{2}}{|1-\langle w,z \rangle |^{2(a+n+1)}} d\mu_{a,b}(w) \\
			=\ds  C \frac{ (1-|z|^{2})^{a+n+1}}{Q_{a,b}(|z|^{2})} \int_{\mathbb{B}_{n}}  \frac{f(\phi_{z}(\zeta))|Q_{a,b}( \langle \phi_{z}(\zeta),z \rangle  )|^{2}}{|1-\langle \phi_{z}(\zeta),z \rangle |^{2(a+n+1)}} (1-| \phi_{z}(\zeta)|^{2})^{a} |\phi_{z}(\zeta) |^{2b} \bigg ( \frac{1-|z|^{2}}{| 1- \langle \zeta , z \rangle |^{2}} \bigg )^{n+1} d\lambda_{n}(\zeta)		
		\end{array}$$
		where $C=\frac{(n-1)!}{\pi^{n} \beta(a+1,b+n)}$. Hence, using Formula \eqref{eq2} and that
		$$|1-\langle \phi_{z}(\zeta),z \rangle |^{2(a+n+1)}= \left(\frac{1-|z|^{2}}{|1- \langle \zeta ,z \rangle|} \right)^{2(a+n+1)},$$
		we obtain
		$$\mathbf{B}_{a,b}f(z)= \frac{C}{Q_{a,b}(|z|^{2})} \int_{\mathbb{B}_{n}} f(\phi_{z}(\zeta)) |Q_{a,b}( \langle \phi_{z}(\zeta),z \rangle  )|^{2} (1-|\zeta |^{2})^{a} | \phi_{z}(\zeta)|^{2b} d\lambda_{n}(\zeta).$$
		Since for every $ \xi \in \mathbb{S}_{n}=\partial\mathbb B_n$, we have $\phi_{z}(\zeta) \to \xi $ as $z \to \xi$, for every $\zeta \in \mathbb{B}_{n}$, then  \begin{equation*}
			\lim_{z \to \xi}\mathbf{B}_{a,b}f(z)= C Q_{a,b}(1) f(\xi) \int_{\mathbb{B}_{n}} (1-|\zeta|^{2})^{a} d\lambda_{n}(\zeta)
		\end{equation*}
		It is easy to see that \begin{equation*}
			\int_{\mathbb{B}_{n}} (1-|\zeta|^{2})^{a} d\lambda_{n}(\zeta)= \frac{\pi ^{n}}{(n-1)!}\beta(a+1,n).
		\end{equation*}
		Using a formula on page 40 in \cite{Ma-Ob-So}, one can see that $$Q_{a,b}(1)=\frac{\Gamma(b+n)\Gamma(a+n+1)}{\Gamma(n)\Gamma(b+a+n+1)}.$$ It follows that \begin{equation*}
			\lim_{z \to \xi}\mathbf{B}_{a,b}f(z)= f(\xi).
		\end{equation*}
	\end{proof}
\section{BMO in Bergman-Poincar\'e metric}
To start, if we set $\psi_{a,b}(z)=\log\mathcal{K}_{a,b}(|z|^{2}) $, then it is easy to see that $\psi_{a,b}$ is plurisubharmonic and smooth  on $\mathbb{B}_{n}$. Hence its Hessian matrix
\begin{equation*}
A(z)= \begin{pmatrix} \ds\frac{\partial^2\psi_{a,b}}{\partial z_1\partial \overline{z}_1 }(z)   & \dots &  \ds\frac{\partial^2\psi_{a,b}}{\partial z_1\partial \overline{z}_n }(z)\\
	\vdots & \ddots & \vdots \\ \ds\frac{\partial^2\psi_{a,b}}{\partial z_n\partial \overline{z}_1 }(z)  & \dots &\ds\frac{\partial^2\psi_{a,b}}{\partial z_n\partial \overline{z}_n }(z)   \end{pmatrix}.	
\end{equation*}
is positive and Hermitian. In particular, we have
$$\langle A(z).\xi,\xi\rangle=\sum_{j,k=1}^n\ds\frac{\partial^2\psi_{a,b}}{\partial z_j\partial \overline{z}_k }(z)\xi_j\overline{\xi}_k\geq 0\,\quad \forall\; \xi\in\mathbb C^n.$$
A simple computation shows that
\begin{equation*}
	\ds\frac{\partial^2\psi_{a,b}}{\partial z_j\partial \overline{z}_k }(z)= \left[\frac{\mathcal{K''}_{a,b}(|z|^{2})}{\mathcal{K}_{a,b}(|z|^{2})} - \left(\frac{\mathcal{K'}_{a,b}(|z|^{2})}{\mathcal{K}_{a,b}(|z|^{2})} \right)^{2}\right]\overline{z}_jz_k + \frac{\mathcal{K'}_{a,b}(|z|^{2})}{\mathcal{K}_{a,b}(|z|^{2})} \delta_{j,k}.
\end{equation*}
So that
\begin{equation}\label{q4.1}
    \langle  A(z).\xi,\xi\rangle= \left[\frac{\mathcal{K}''_{a,b}(|z|^{2})}{\mathcal{K}_{a,b}(|z|^{2})} - \left(\frac{\mathcal{K}'_{a,b}(|z|^{2})}{\mathcal{K}_{a,b}(|z|^{2})} \right)^{2}\right] |\langle z,\xi\rangle|^2 +\frac{\mathcal{K'}_{a,b}(|z|^{2})}{\mathcal{K}_{a,b}(|z|^{2})}|\xi|^2,\quad \forall\; (z,\xi)\in\mathbb
 B_n\times \mathbb C^n.
\end{equation}
It follows that we can define a metric on $\mathbb B_n$ (called Bergman-Poincar\'e metric) as follows; for any two points $z$ and $w$ in $\mathbb{B}_{n}$, let $\mathbf{d}_{a,b}(z,w)$ be the infimum of lengths $\ell_{a,b}(\gamma)$ of piecewise smooth curves  $\gamma: [0,1] \to \mathbb{B}_{n}$ with $\gamma(0)=z$ and $\gamma(1)=w$, where
\begin{equation*}
	\ell_{a,b}(\gamma)= \int_{0}^{1} \langle  A(\gamma(t)) \gamma'(t), \gamma'(t)\rangle^{1/2} dt.
\end{equation*}
The space $(\mathbb B_n,\mathbf{d}_{a,b})$ is a complete metric space.\\

Now, we can introduce BMO functions.
\begin{defn}
    For every $f\in L^2(\mathbb{B}_{n}, d\mu_{a,b})$,  the mean oscillation of $f$ at $z\in\mathbb B_n$ is \begin{equation*}
	MO(f)(z)=\left(\mathbf{B}_{a,b}(|f|^{2})(z)-|\mathbf{B}_{a,b}f(z)|^{2}\right)^{\frac{1}{2}}
\end{equation*}
We say that $f\in BMO$ if
$$\|f\|_{BMO}:=\sup_{z\in\mathbb B_n}MO(f)(z)<+\infty.$$
\end{defn}
The square root in the previous definition has a sense since \begin{equation*}
	(MO(f)(z))^{2}=\frac{1}{2 \mathcal{K}_{a,b}^{2}(|z|^{2})} \int_{\mathbb{B}_{n}}\int_{\mathbb{B}_{n}} |f(u)-f(v)|^{2} |\mathcal{K}_{a,b}( \langle z,u \rangle) |^{2} |\mathcal{K}_{a,b}( \langle z,v \rangle) |^{2} d\mu_{a,b}(u) d\mu_{a,b}(v).
\end{equation*}
Our aim here is to extend a result due to B\'ekoll\'e et al. in \cite{BBCZ} to our setting. Namely, we will prove the following theorem:
\begin{thm}\label{th2}
    For every $f\in BMO$, we have
    $$|\mathbb B_{a,b}f(z)-\mathbb B_{a,b}f(w)|\leq 2\|f\|_{BMO}\mathbf{d}_{a,b}(z,w)
    $$
    for all $z,w\in\mathbb B_n$.
\end{thm}
The crucial step of the proof is the following proposition. In what follows, for every $z\in \mathbb{B}_{n}$ we consider the orthogonal projection  $\mathscr P_{z}$  from $\mathcal{A}^{2}_{a,b}(\mathbb{B}_{n},d\mu_{a,b})$ onto the one-dimensional subspace spanned by the normalized reproducing kernel $\kappa_{z}$. Therefore \begin{equation*}
	\mathscr P_{z}f(\zeta)=\langle f,\kappa_{z} \rangle \kappa_{z}(\zeta)=\frac{\mathcal{K}_{a,b}(\langle \zeta,z\rangle)}{\mathcal{K}_{a,b}(|z|^{2})} f(z).
\end{equation*}
\begin{prop}\label{p4.2}
	Let $\gamma(t)$ be a smooth curve in $\mathbb{B}_{n}$. Then \begin{equation*}
		\left\| (I-\mathscr P_{\gamma(t)}) \left( \frac{d}{dt}  \kappa_{\gamma(t)} \right) \right\|^2_{2,a,b}=\langle A(\gamma(t))\gamma'(t), \gamma'(t)\rangle.
	\end{equation*}
\end{prop}
\begin{proof}
	Since we have \begin{equation*}
		\kappa_{z}(\zeta)=\frac{\mathcal{K}_{a,b}(\langle \zeta,z\rangle)}{\sqrt{\mathcal{K}_{a,b}(|z|^{2})}}
	\end{equation*}
	then \begin{equation}\label{eq3}
		\frac{d}{dt}  \kappa_{\gamma(t)}(\zeta) = \frac{\langle \zeta,\gamma'(t) \rangle \mathcal{K'}_{a,b}(\langle \zeta,\gamma(t) \rangle)}{\mathcal{K}^{\frac{1}{2}}_{a,b}(|\gamma(t)|^{2})}- \frac{\Re\langle \gamma'(t),\gamma(t)\rangle \mathcal{K'}_{a,b}(|\gamma(t)|^{2})}{\mathcal{K}^{\frac{3}{2}}_{a,b}(|\gamma(t)|^{2})} \mathcal{K}_{a,b}(\langle \zeta,\gamma(t) \rangle).
	\end{equation}
	It follows that
    \begin{align} \label{eq4}
		\mathscr P_{\gamma(t)} \bigg ( \frac{d}{dt}  \kappa_{\gamma(t)} \bigg)(\zeta)&=\frac{ \mathcal{K}_{a,b}(\langle \zeta,\gamma(t) \rangle)}{\sqrt{\mathcal{K}_{a,b}(|\gamma(t)|^{2})}} \frac{d}{dt}  \kappa_{\gamma(t)}(\gamma(t))\\ \nonumber
		\ & = \frac{\mathcal{K}_{a,b}(\langle \zeta,\gamma(t) \rangle) \mathcal{K'}_{a,b}(|\gamma(t)|^{2})}{\mathcal{K}^{\frac{3}{2}}_{a,b}(|\gamma(t)|^{2})} \bigg(\langle \gamma(t),\gamma'(t) \rangle - \Re \langle \gamma'(t), \gamma(t) \rangle\bigg).
    \end{align}
	Using \eqref{eq3} and \eqref{eq4}, we obtain \begin{align*}
		(I-\mathscr P_{\gamma(t)}) \bigg ( \frac{d}{dt}  \kappa_{\gamma(t)} \bigg) & = \frac{\langle \zeta,\gamma'(t) \rangle) \mathcal{K'}_{a,b}(\langle \zeta,\gamma(t) \rangle)}{\mathcal{K}^{\frac{1}{2}}_{a,b}(|\gamma(t)|^{2})} - \frac{\mathcal{K}_{a,b}(\langle \zeta,\gamma(t) \rangle)\mathcal{K'}_{a,b}(|\gamma(t)|^{2}) }{\mathcal{K}^{\frac{3}{2}}_{a,b}(|\gamma(t)|^{2})} \\ \nonumber
		\ & =: H_{t}(\zeta)- G_{t}(\zeta).
    \end{align*}
    Thus \begin{equation*}	
		\left\| (I-\mathscr P_{\gamma(t)}) \left( \frac{d}{dt}  \kappa_{\gamma(t)} \right) \right\|^{2}_{2,a,b}= \| H_{t}\|_{a,b,2}^{2}+\|G_{t}\|^{2}_{a,b,2}-2 \Re \langle H_{t}, G_{t} \rangle_{a,b}.
	\end{equation*}
	It's easy to see that \begin{equation}\label{4.3}
		\|G_{t}\|^{2}_{a,b,2}= \langle H_{t}, G_{t} \rangle_{a,b}= \bigg (\frac{\mathcal{K'}_{a,b}(|\gamma(t)|^{2})}{\mathcal{K}_{a,b}(|\gamma(t)|^{2})} \bigg)^{2} |\langle \gamma(t), \gamma'(t) \rangle |^{2},
	\end{equation}
	where we use the fact that $\zeta \mapsto \langle \zeta, \gamma'(t) \rangle \mathcal{K'}_{a,b}(\langle \zeta,\gamma(t) \rangle)$ is in $\mathcal{A}_{a,b}^{2}(\mathbb{B}_{n}, d\mu_{a,b})$. Hence \begin{equation*}
			\left\| (I-\mathscr P_{\gamma(t)}) \left( \frac{d}{dt}  \kappa_{\gamma(t)} \right) \right\|^{2}_{2,a,b}= \| H_{t}\|_{a,b,2}^{2}-\|G_{t}\|^{2}_{a,b,2}.
	\end{equation*}
	Using both the polar coordinates and orthogonality, we can obtain the result below
	\begin{align*}
			\|H_{t}\|^{2}_{a,b,2}&=   \frac{1}{\mathcal{K}_{a,b}(|\gamma(t)|^{2})}\int_{\mathbb{B}_{n}} | \langle \zeta , \gamma'(t) \rangle |^{2} | \mathcal{K'}_{a,b}(\langle \zeta,\gamma(t) \rangle) |^{2} d\mu_{a,b}(\zeta) \\ \nonumber
			\ &= \frac{C(a,b,n)}{\mathcal{K}_{a,b}(|\gamma(t)|^{2})} \int_{0}^{1}  r^{2+2b+2n-1}(1-r^{2})^{a}  \int_{\mathbb{S}^{n}} | \langle \xi , \gamma'(t) \rangle |^{2} |\mathcal{K'}_{a,b}(r \langle \xi,\gamma(t) \rangle)|^{2}dr d\sigma_{n}(\xi)
	\end{align*}
where $C(a,b,n)= \frac{(n-1)!}{\pi^{n}\beta(a+1,b+n)}$ as used before. \\
	Let \begin{equation*}
		I(t)= \int_{\mathbb{S}^{n}} | \langle \xi , \gamma'(t) \rangle |^{2} |\mathcal{K'}_{a,b}(r \langle \xi,\gamma(t) \rangle)|^{2} d\sigma_{n}(\xi)
	\end{equation*}
	and Set
 \begin{equation*}
		v_{1}(t)=\frac{\gamma(t)}{|\gamma(t)|}, \qquad v_{2}(t)=\frac{\phi(t)}{| \phi(t)|}
	\end{equation*}
where $\phi(t)=\gamma'(t)-\langle \gamma'(t),v_{1}(t) \rangle v_{1}(t)$ and completing $v_3, \dots,v_n$ to  obtain an orthonormal basis  of $\mathbb{C}^{n}$. In this basis, we have $u_{1}=\langle \xi , v_{1}(t) \rangle$ is the first coordinate and  $u_{2}=\langle \xi , v_{2}(t) \rangle$ is the second one, hence we obtain \begin{equation*}
	\langle \xi , \gamma(t) \rangle= |\gamma(t)| \langle \xi , v_{1}(t)\rangle=|\gamma(t)|u_{1}  \qquad\text{and} \qquad \langle \xi , \gamma'(t)\rangle= \langle v_{1}(t),\gamma'(t) \rangle u_{1}+|\phi(t)|u_{2}.
\end{equation*}
By a change of variables, we find \begin{equation*}
		I(t)=\int_{\mathbb{S}^{n}} |  \langle v_{1}(t),\gamma'(t) \rangle u_{1}+|\phi(t)|u_{2}  |^{2} |\mathcal{K'}_{a,b}(r|\gamma(t)|u_{1})|^{2} d\sigma_{n}(u).
\end{equation*}
It follows that \begin{align*}
I(t)&=  \sum_{k,j=0}^{+\infty} \bigg (\frac{n(a+b+n+1)}{b+n} \bigg)^{2} \frac{(n+1)_{k}(n+1)_{j}(a+b+n+2)_{k}(a+b+n+2)_{j}}{(b+n+1)_{k}(b+n+1)_{j}} r^{k+j} |\gamma(t)|^{k+j} \\ \nonumber
	\ & \times \int_{\mathbb{S}^{n}} \bigg ( |\langle  v_{1}(t),\gamma'(t) \rangle |^{2} |u_{1}|^{2}+|\phi(t)|^{2}|u_{2}|^{2}- 2|\phi(t)|\Re(\langle v_{1},\gamma'(t)\rangle u_{1}\overline{u}_2)\bigg) u_{1}^{k} \overline{u_{1}}^{j} d\sigma_{n}(u).
\end{align*}
Using now the orthogonality and \eqref{eq1}  , we can deduce that \begin{align*}
		\|H_{t}\|^{2}_{a,b,2}&= \frac{(n-1)!}{\pi^{n}\beta(a+1,b+n)} 2\pi^{n} \bigg (\frac{n(a+b+n+1)}{b+n} \bigg)^{2} \frac{1}{\mathcal{K}_{a,b}(|\gamma(t)|^{2})} \sum_{k=0}^{+\infty} \frac{(n+1)_{k}^{2}(a+b+n+2)_{k}^{2}}{(b+n+1)_{k}^{2} k!} |\gamma(t)|^{2k} \\ \nonumber
		 \ & \times\bigg [ \frac{k}{|\gamma(t)|^{2}}|\langle \gamma'(t),\gamma(t)\rangle|^{2}+|\gamma'(t)|^{2}\bigg] \int_{0}^{1} r^{2k+2b+2n+1}(1-r^{2})^{a} dr \\ \nonumber
		 \ & = \frac{n(a+b+n+1)}{\mathcal{K}_{a,b}(|\gamma(t)|^{2}) (b+n)} \sum_{k=0}^{+\infty} \frac{(n+1)_{k}(a+b+n+2)_{k}}{(b+n+1)_{k} k!} |\gamma(t)|^{2k} \bigg [ \frac{k}{|\gamma(t)|^{2}}|\langle \gamma'(t),\gamma(t)\rangle|^{2}+|\gamma'(t)|^{2}\bigg].
\end{align*}
This is exactly \begin{equation}\label{4.4}
	\|H_{t}\|^{2}_{a,b,2}= \frac{\mathcal{K''}_{a,b}(|\gamma(t)|^{2})}{\mathcal{K}_{a,b}(|\gamma(t)|^{2})} |\langle \gamma'(t), \gamma(t) \rangle|^{2} + \frac{\mathcal{K'}_{a,b}(|\gamma(t)|^{2})}{\mathcal{K}_{a,b}(|\gamma(t)|^{2})}|\gamma'(t)|^{2}
\end{equation}
Using \eqref{4.3}, \eqref{4.4} and \eqref{q4.1}, we conclude that $$\begin{array}{l}
		\ds\left\| (I-\mathscr P_{\gamma(t)}) \left( \frac{d}{dt}  \kappa_{\gamma(t)} \right) \right\|^{2}_{2,a,b} = \ds \| H_{t}\|_{a,b,2}^{2}-\|G_{t}\|^{2}_{a,b,2} \\
		=\ds \left[\frac{\mathcal{K''}_{a,b}(|\gamma(t)|^{2})}{\mathcal{K}_{a,b}(|\gamma(t)|^{2})} -\left (\frac{\mathcal{K'}_{a,b}(|\gamma(t)|^{2})}{\mathcal{K}_{a,b}(|\gamma(t)|^{2})} \right)^{2} \right]|\langle \gamma(t), \gamma'(t) \rangle |^{2} + \frac{\mathcal{K'}_{a,b}(|\gamma(t)|^{2})}{\mathcal{K}_{a,b}(|\gamma(t)|^{2})} |\gamma'(t)|^{2}  \\
		 =\ds \langle A(\gamma(t))\gamma'(t), \gamma'(t)\rangle.
\end{array}$$
\end{proof}
\begin{lem}\label{lem4.3}
	Let $f \in BMO$ and  $\gamma:[0,1]\longrightarrow \mathbb B_n$ be a smooth curve. Then \begin{equation}
			\left| \frac{d}{dt}\mathbf B_{a,b}f(\gamma(t))  \right| \leq  2 	MO(f)(\gamma(t))  \langle A(\gamma(t))\gamma'(t), \gamma'(t)\rangle^{\frac{1}{2}}.
	\end{equation}
	\begin{proof}
		Differentiation under the integral sign gives
		\begin{equation*}
			\frac{d}{dt}\mathbf B_{a,b}f(\gamma(t)) = 2 \int_{\mathbb{B}_{n}} f(w) \Re \left[ \frac{d}{dt}  \kappa_{\gamma(t)} (w)  \overline{\kappa_{\gamma(t)}(w)} \right] d\mu_{a,b}(w).
		\end{equation*}
	After making a few simple modifications, we can use the same argument as in the proof of Theorem 2.22 in \cite{He-Ko-Zh}, to deduce that
		\begin{align*}
			\left | \frac{d}{dt}\mathbf B_{a,b}f(\gamma(t))  \right| & \leq  2 \int_{\mathbb{B}_{n}} \left | f(w) - \mathbf{B}_{a,b}f(\gamma(t)) \right| | \kappa_{\gamma(t)}(w)| \left |  (I-\mathscr P_{\gamma(t)}) \left( \frac{d}{dt}  \kappa_{\gamma(t)} \right)(w) \right|  d\mu_{a,b}(w) \\ \nonumber
			\ & \leq 2 \| f-\mathbf B_{a,b}f(\gamma(t))\kappa_{\gamma(t)}  \|_{2,a,b} \left\| (I-\mathscr P_{\gamma(t)}) \left( \frac{d}{dt}  \kappa_{\gamma(t)} \right) \right\|_{2,a,b} \\ \nonumber
			\ & \leq 2 	MO(f)(\gamma(t))  \langle A(\gamma(t))\gamma'(t), \gamma'(t)\rangle^{\frac{1}{2}}
		\end{align*}
	where we use the fact that \begin{equation*}
		\| f-\mathbf{B}_{\alpha,\beta}f(\gamma(t))\kappa_{\gamma(t)}  \|_{2,a,b}=MO(f)(\gamma(t)).
	\end{equation*}
	\end{proof}
\end{lem}
	Now we can deduce the proof of Theorem \ref{th2}:
	\begin{proof}
		Let $f\in BMO$ and $z,w\in\mathbb B_n$. Let $\gamma:[0,1]\longrightarrow \mathbb B_n$ be a piecewise smooth curve with $\gamma(0)=z$ and $\gamma(1)=w$. Without loss of generality, we can assume that $\gamma$ is smooth on $[0,1]$.  By Lemma \ref{lem4.3},
		$$\begin{array}{lcl}
			\ds |\mathbf{B}_{a,b}f(z)-\mathbf{B}_{a,b}f(w)|&=&\ds \left|\int_0^1 \frac{d}{dt}\mathbf{B}_{a,b}f(\gamma(t))dt\right|\\
        &\leq&\ds 2\int_0^1	MO(f)(\gamma(t))  \langle A(\gamma(t))\gamma'(t), \gamma'(t)\rangle^{\frac{1}{2}}dt \\
			&\leq&\ds 2\|f\|_{BMO}\int_0^1 \langle A(\gamma(t))\gamma'(t), \gamma'(t)\rangle^{\frac{1}{2}}dt\\
        &\leq&\ds2\|f\|_{BMO}\ell_{a,b}(\gamma).
		\end{array}$$
		To conclude the result, it suffices to take the infimum over all piecewise smooth curves $\gamma$.
	\end{proof}
 \section{Comments and concluding remarks}
 In this concluding section, we discuss the zero set $\mathcal Z_{Q_{a,b}}$ of $Q_{a,b}$ in $\mathbb D$. Indeed, the condition  $\mathcal Z_{Q_{a,b}}=\emptyset$ was used in many results. We claim that the function $Q_{a,b}$ is a particular case of hypergeometric functions studied in many papers (see for examples papers of Duren and his collaborators in \cite{Bo-Du, Dr-Du1, Dr-Du2, Dr-Du3, Dr-Mo, Du-Gu}). Until now there is no complete answer for these functions. To clarify the situation, we will restrict our work to the case $n=1$ to use a result due to Duren and Guillou \cite{Du-Gu} in the last situation when  $a$ is a large integer and $b$ is proportional to $a$.\\
 We inspect first the relationship between zeros of $Q_{a,b}$ and those of the modified Bergman kernel $\mathbb K_{a,b}$, then we look into the cases when $b$ is near 0 or $-n$ and finally we consider the case $b$ large enough.
 \begin{enumerate}
     \item Since $Q_{a,b}$ is holomorphic on $\mathbb D$ then  $\mathcal Z_{Q_{a,b}}$ is a discreet set.\\
 Moreover, since $$\mathbb K_{a,b}(z,w)=\mathcal K_{a,b}(\langle z,w\rangle)=\frac{Q_{a,b}(\langle z,w\rangle)}{(1-\langle z,w\rangle)^{a+n+1}}$$
         then the set of zeros of $\mathbb K_{a,b}$ is given by $$\bigcup_{\xi\in \mathcal Z_{Q_{a,b}}}\left\{(z,w)\in\mathbb B_n\times \mathbb B_n;\ \langle z,w\rangle=\xi\right\}=:\bigcup_{\xi\in \mathcal Z_{Q_{a,b}}}A_\xi$$
         a discreet union of sets $A_\xi$ where each $A_\xi$ is a submanifold  of $\mathbb R^{4n}$ with real dimension $4n-2$. Moreover, $A_\xi$ is a Cauchy-Riemann submanifold of $\mathbb R^{4n}$ with Cauchy-Riemann dimension $4n-4$ (see \cite{Bo} for more details about Cauchy-Riemann submanifolds).
   \item  We claim that the set $\mathcal Z_{Q_{a,b}}$ can be empty, finite, or infinite according to the values of $a$ and $b$. Indeed, using Rouch\'e theorem, one can easily show that there exist two real constants $-n<b_1(a)<b_2(a)<0$ depending on $a$ such that for every $b\in]-n,b_1(a)[$, the function $Q_{a,b}$ has at most $n$ (simple) zeros in  $\mathbb D$ and it has no zero in $\mathbb D$ for every $b\in ]b_2(a),-b_2(a)[$ (see Figure \ref{fig1}).
   \item We claim that if $b=-n+k$ where $k\in\mathbb N$ such that $k<n$ then $Q_{a,b}$ is a polynomial with degree at most $n-k$. Thus it has at most $n-k$ zeros in $\mathbb D$. By Taking for example $b=-1$, then we obtain $Q_{a,-1}(z)=1+\frac{a+1}{n-1}z$. It follows that the unique zero of $Q_{a,-1}$ in $\mathbb C$ is $-\frac{n-1}{a+1}$.

   \item Now for $b$ large enough, namely if  $n=1,\ a=m-1$ and $b=\tau m+1$ for some integer $m$ and a parameter $\tau>0$, then thanks to Duren and Guillou \cite{Du-Gu} (say Boggs and Duren \cite{Bo-Du} when $\tau\in\mathbb N$), the zeros (in $\mathbb C$) of the hypergeometric polynomial $Q_{m-1,\tau m+1}$ cluster on the loop of the lemniscate
   $$\left\{z\in\mathbb C;\ \left|z^\tau(z-1)\right|=\frac{\tau^\tau}{(1+\tau)^{1+\tau}}\right\}\quad \text{with}\quad \Re z>\frac{\tau}{\tau +1}$$
    when $m\to+\infty$. Hence the set $\mathcal Z_{Q_{a,b}}$ can be infinite in $\mathbb D$. (see figure \ref{fig2}).\\
It is interesting to study $\mathcal Z_{Q_{a,b}}$ for $b$ large and $n\geq2$. The well-known methods that can be used here are the quadratic differentials and their short trajectories or the saddle point method. We think that the second method is more suitable for this case. Each one of these methods requires many computations that cannot be developed in this paper. It may be the main subject of a new paper.
 \end{enumerate}

\begin{figure}
  \includegraphics[scale=0.5]{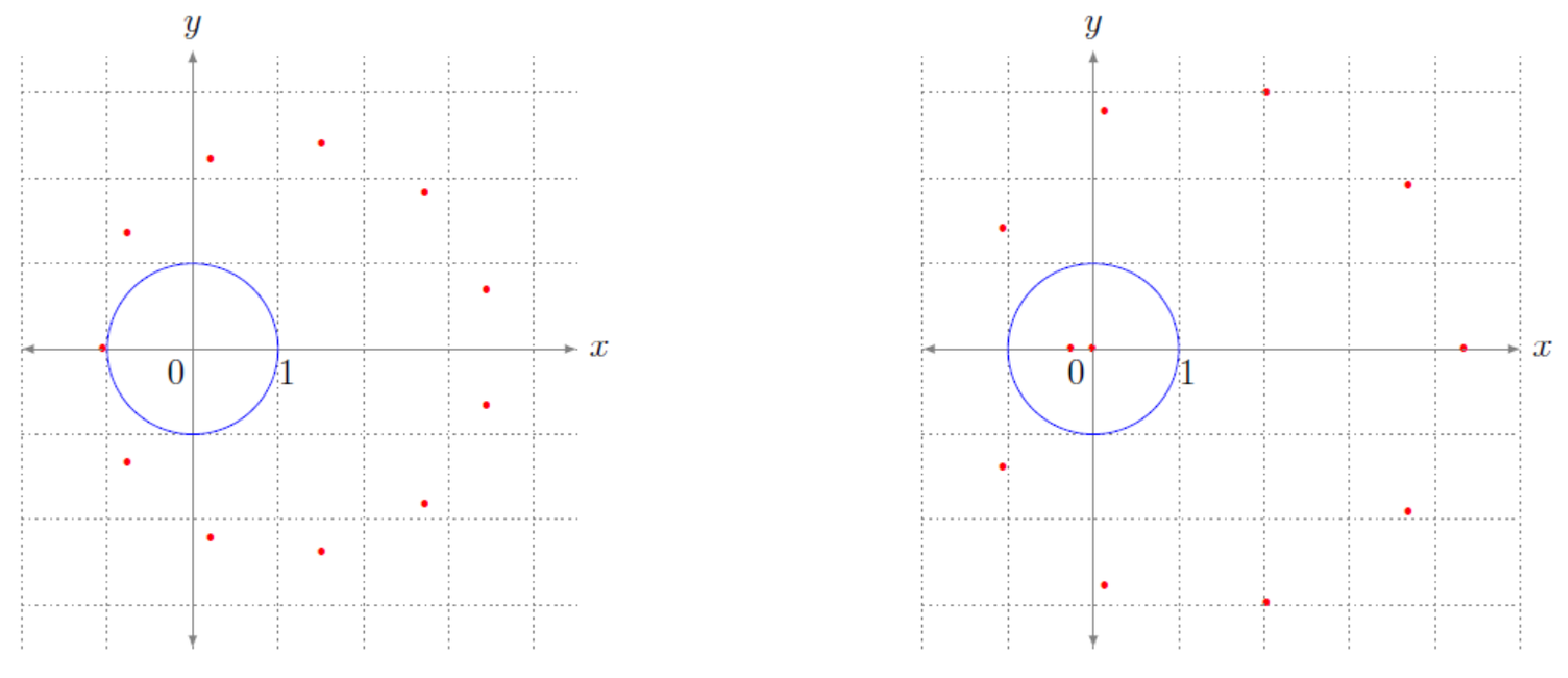}\\
  \caption{Zeros of $Q_{10,b}$ with $n=2$ and $b=-0.01$ at left and $b=-1.9$ at right.}\label{fig1}
\end{figure}

\begin{figure}
  \includegraphics[scale=0.5]{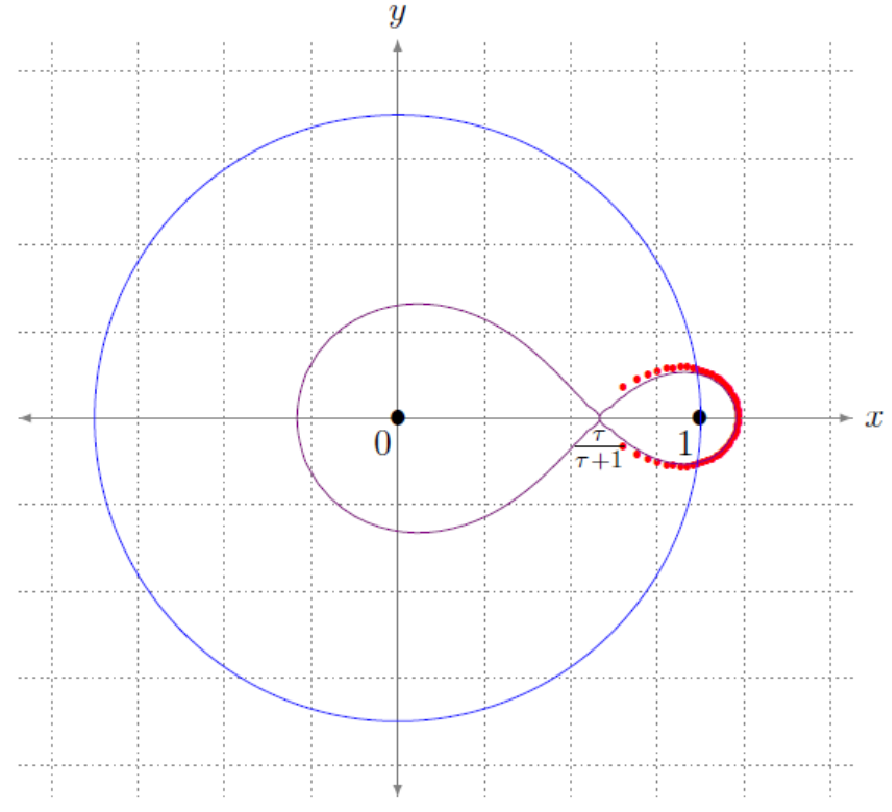}\\
  \caption{Zeros of the polynomial $Q_{m-1,\tau m+1}$ with $n=1,\ m=50$ and $\tau=2$.}
    \label{fig2}
\end{figure}

\end{document}